\documentclass[11pt]{article}
\usepackage[margin=2.6cm]{geometry}
\usepackage{amsmath,amssymb,amsthm}
\usepackage{enumitem}
\usepackage{cite}
\usepackage[colorlinks=true,linkcolor=blue,citecolor=blue,urlcolor=blue]{hyperref}

\numberwithin{equation}{section}
\theoremstyle{plain}
\newtheorem{theorem}{Theorem}[section]
\newtheorem{proposition}[theorem]{Proposition}
\newtheorem{lemma}[theorem]{Lemma}
\newtheorem{corollary}[theorem]{Corollary}
\theoremstyle{definition}
\newtheorem{definition}[theorem]{Definition}
\newtheorem{example}[theorem]{Example}
\newtheorem{remark}[theorem]{Remark}

\title{Exchangeable testing against an unknown benchmark}
\author{Alexander Gnedin\thanks{School of Mathematical Sciences, Queen Mary University of London, London E1 4NS, United Kingdom; \texttt{a.gnedin@qmul.ac.uk}}}
\date{}

\begin{document}
\maketitle

\begin{abstract}
We generate infinite binary exchangeable sequences by sequential comparison of data points against a latent benchmark. Assuming a prior distribution of the benchmark rank \(R_0\) within an unobserved group of players, we set up the Bayesian machinery which determines the posterior distribution of the running rank \(R_n\) in purely combinatorial terms, by counting insertion slots in a uniformly random ranking. This yields an explicit predictive probability of winning against the benchmark, which is a rational function of the counts. We identify some cases where the predictive probability has particularly simple form, notably a tournament related to the Topp--Leone distribution.
The running benchmark rank is a P\'olya urn with random initial composition, and its normalised value converges to a latent strength variable with polynomial density. In a projective time scale, the fluctuations of the normalised rank
about this limit are Brownian, with the Wright--Fisher coefficient frozen at the limit, while the associated recurrent up-down Markov chain on a high  level $n$  
of  the Pascal graph  is approximated by a Wright--Fisher diffusion with a nonlinear drift expressed through the prior density.


\end{abstract}


\section{Introduction}\label{sec:intro}

For an exchangeable sequence of binary trials, let \(\pi(A,B)\) denote the predictive probability that the next trial is a success, given \(A\) successes and \(B\) failures in the trials observed so far. Laplace's \cite{Laplace1774} Rule of Succession for Bernoulli trials with a uniformly distributed success probability is
\begin{equation}\label{RoS}
\pi(A,B)=\frac{A+1}{A+B+2}.
\end{equation}
The uniform prior was introduced by  Bayes \cite{Bayes1763},
who advocated its noninformativeness by
  showing that the number of successes in $n$ trials  has the discrete uniform distribution on $\{0,\ldots,n\}$.
 The value \(\pi(n,0)\), the probability that a run of \(n\) successes is continued, made the rule famous and controversial, triggering a century-long debate around the sunrise prediction problem \cite{Zabell1989, Zabell2005}.

Markov \cite{Markov1917}  studied a two-colour urn scheme, where a ball is randomly sampled from the composition $(A,B)$ and returned together with one more ball of the same colour. 
The model became widely known through Eggenberger and P\'olya, see \cite{JohnsonKotz1977} for many variations and history.
We will call this process the basic Markov-P{\'o}lya urn.
Following de Finetti \cite{deFinetti1952},
we may identify the process with a transient Markov chain on the Pascal graph -- the lattice ${\mathbb Z}_{\geq 0}^2$ with edges pointing away from the origin -- then 
for the initial composition $(1,1)$ the transition probability after $n$ steps 
coincides with (\ref{RoS}), and for any other  start state it only differs  by  constant shift of the counts.
This connection with the Rule of Succession makes transparent the exchangeability of  increments of the chain along with the  distinguished  role of the family of  Beta  distributions.

By de Finetti's representation theorem an infinite exchangeable binary sequence is a mixture of independent Bernoulli trials, and \(\pi(A,B)\) is a ratio of binomial moments of the mixing prior. Explicit examples might therefore be expected in abundance. Yet the theory of discrete exchangeable structures 
rests almost entirely on processes where the predictive probabilities
 are ratios of linear functions of the counts; in the binary case 
 this necessarily forces a Beta mixing.
Haldane's prior, the equal mixture of the uniform distribution and the point mass at \(1\), gives a nonlinear \(\pi(n,0)\) \cite{Haldane1932, WagenmakersZabellGronau2024}, but nonlinear rules explicit in the full state \((A,B)\) are rare; an early instance
 is due to Daboni \cite{Daboni1953}. Many  urn schemes have nonlinear reinforcement, 
but they are not exchangeable. 

In this paper we 
replace  the  preferential attachment paradigm of Markov-P{\'o}lya urn processes
 by comparison with a hidden benchmark, very much in the spirit of Bayes's \cite{Bayes1763} splitting 
of  uniform sample at a random point.
 An initial group of \(k\) past players is ranked uniformly at random, and a choice operator, depending on the ranks only, selects one of them as the benchmark. New players arrive and each of them is tested against the benchmark. The observer knows the  initial distribution  of the benchmark rank and sees the outcomes of the tests; the posterior law of the rank follows then 
from the independence properties of Laisant's factorial code of permutations \cite{Laisant1888}, also known as the Lehmer code \cite{Lehmer1960}. Concretely, the running rank of the benchmark evolves like a P\'olya urn with random initial composition, thus
the benchmark density is a mixture of Beta densities with integer parameters (a Bernstein polynomial).

Mixed Beta priors are classical in Bayesian prediction \cite{FortiniPetrone2024}. A novel element introduced here is the mechanism: the mixture weights are the law of an integer rank produced by a choice operator, and they are updated on random permutations, with no continuous parameter involved. 
For $k=1$ the running benchmark rank has uniform distribution, which reduces (\ref{RoS}) to a distribution-free property of order statistics and
well known combinatorial arguments \cite{Feller1971}, which we adapt to a more complex setting.

A simple example shows the flavour. Four players are arranged in pairs, the stronger player of each pair is retained, and the weaker of the two pairwise winners is the benchmark. The rank of the benchmark among the four is \(2\) or \(3\) with probabilities \(\tfrac23\) and \(\tfrac13\), and after \(A\) wins and \(B\) losses the next player wins with probability
\begin{equation}\label{TL2}
\pi(A,B)=\frac{(A+2)(A+2B+7)}{(A+B+5)(A+2B+6)}.
\end{equation}
The normalised running rank converges,  as the number of tests grows, to a random variable with the Topp--Leone distribution 
\cite{KotzVandorp2004}.


For large samples the choice of prior is largely immaterial: for every prior with a positive continuous density the predictive odds are  close to \(A:B\) when both counts are large. 
The interest of explicit rules lies in small and moderate samples, and in the dynamics of the latent variables. The latter is visible in diffusion approximations: we will show that the running rank approaches its limit like a Brownian motion in a projective time scale, and that the up-down chain on $n$th level of the Pascal graph is approximated by a Wright--Fisher diffusion with nonlinear drift determined by the prior.

The rest of the paper is organised as follows. Section~\ref{sec:model} introduces the model and Section~\ref{sec:update} the updating rules. Section~\ref{sec:ident} deals with identifiability of the rank prior, Section~\ref{sec:tilt} with Beta tilting, and Section~\ref{sec:examples} with predictive rules of fixed degree, including the factorisation theorem. Choice operators and the Topp--Leone family are discussed in Sections~\ref{sec:choice} and~\ref{sec:TL}, the limit theorems in Section~\ref{sec:diff}, finite exchangeability in Section~\ref{sec:finite}, and improper priors and partially exchangeable statistics of the model in Section~\ref{sec:general}. The formulas in the examples were all checked by computer algebra.

\section{The binary testing model}\label{sec:model}

Let an unobserved group of \(k\) past players labelled \(P_{-1},\dots,P_{-k}\) have i.i.d.\ \(\mathrm{U}[0,1]\) latent marks, a smaller mark signifying a stronger player. Denote the ordered marks \(V_1<\dots<V_k\). The ranking of \(P_{-1},\dots,P_{-k}\) is uniformly random and independent of the ordered marks, and the marginal densities are
\[
V_r\stackrel{d}{=}\mathrm{Beta}(r,k-r+1),\qquad
h_{k,r}(x):=k\binom{k-1}{r-1}x^{r-1}(1-x)^{k-r},\quad x\in[0,1].
\]

Suppose a choice procedure whose outcome only depends on the ranks of the players identifies a winner with rank \(R_0\). The strength of the winner,
\[
X:=V_{R_0},
\]
will be called the benchmark  location. With prior rank distribution
\(p_0(r):=\mathbb{P}[R_0=r]\), \(r\in[k]\), the unconditional  density of the benchmark is
\begin{equation}\label{eq:f0}
f_0(x):=\sum_{r=1}^k p_0(r)\,h_{k,r}(x).
\end{equation}

Suppose further players \(P_1,P_2,\dots\) with i.i.d.\ \(\mathrm{U}[0,1]\) strengths \(U_1,U_2,\dots\), independent of the past group, arrive sequentially and play against the benchmark:
\begin{equation}\label{eq:win}
P_n\text{ wins}\quad\Longleftrightarrow\quad U_n<X.
\end{equation}
The events \eqref{eq:win} are exchangeable. The prior information of the observer is \(p_0\), and the outcomes are observed sequentially. 

Throughout, \(A\) and \(B\) denote the numbers of wins and losses among \(P_1,\dots,P_n\), \(n=A+B\).  In a standard  way, the binomial moments
\[
m(A,B):=\mathbb{E}\bigl[X^A(1-X)^B\bigr]=\int_0^1x^A(1-x)^Bf_0(x)\,{\rm d}x
\]
are the probabilities of any particular sequence of \(A\) wins and \(B\) losses, and the predictive probability of a win is
\begin{equation}\label{eq:pidef}
\pi(A,B):=\mathbb{P}[P_{n+1}\text{ wins}\mid A,B]=\frac{m(A+1,B)}{m(A,B)},
\end{equation}
which defines a transient Markov chain on the Pascal graph.
The gist of the model is that \(\pi\) can be computed without integration, in terms of the ranks.

The key for the prior-to-posterior analysis is the following combinatorial fact, resulting from Laisant's factorial
code of permutations \cite{Laisant1888} (also known as Lehmer's or inversion code \cite{Lehmer1960}):
\begin{quote}
\textbf{The ranking rule:} the rank of \(P_n\) relative to \(P_{-1},\dots,P_{-k},P_1,\dots,P_n\) has uniform distribution on \(\{1,\dots,n+k\}\); moreover, for \(n=1,2,\dots\) these relative ranks are independent.
\end{quote}

\begin{example}\label{ex:slots}
Let \(w_1,w_2,w_3\) be odds representing the law of \(R_0\) for \(k=3\). The relative rank of \(P_1\) is equally likely to settle in any of four slots; for instance, the configuration
\(w_1,*,w_2,w_3\)
contributes odds \(\tfrac14(w_2+w_3)\) to the winning event. Summing over the slots and normalising, \(P_1\) wins with probability \(\mathbb{E}R_0/4\); for a group of size \(k\), with probability \(\mathbb{E}R_0/(k+1)\). For \(k=1\), after \(A\) wins and \(B\) losses the \(n+2\) slots are equally likely and \(A+1\) of them lie on the winning side of the benchmark, so we are back to 
(\ref{eq:pidef}).
\end{example}

\begin{remark}
The construction of an infinite factorial code  in terms of i.i.d.\ marks first appeared in the context of the secretary problem \cite{Rubin1966}. This way, the ranking rule sets up an isomorphism between the space of sequences of relative ranks of \(P_{-1},\dots,P_{-k},P_1,P_2,\dots\) with the product of discrete uniform distributions and the infinite cube \([0,1]^\infty\) (see \cite{Gnedin2026size} for infinite size-biased models of ranking). 
\end{remark}

\section{Updating rules}\label{sec:update}

\subsection{The running rank}

Let \(R_n\) denote the {\it running rank} of the benchmark relative to \(P_{-1},\dots,P_{-k},P_1,\dots,P_n\). After \(A\) wins and \(B=n-A\) losses,
\begin{equation}\label{eq:Rn}
R_n=R_0+A.
\end{equation}

\begin{proposition}
[The running rank is an urn process]
\label{prop:urn}
Conditionally on \(R_0,\dots,R_n\), the player \(P_{n+1}\) wins with probability \(R_n/(k+n+1)\). Thus \((R_n,\;k+n+1-R_n)_{n\ge0}\) is the Markov--P\'olya urn with random initial composition \((R_0,k+1-R_0)\).
\end{proposition}

\begin{proof}
By the ranking rule, the rank of \(P_{n+1}\) among the \(k+n+1\) players is uniform on \(k+n+1\) slots and independent of the past relative ranks; exactly \(R_n\) of the slots lie above the benchmark.
\end{proof}

The posterior distribution of the running rank is
\begin{equation}\label{eq:post}
\mathbb{P}[R_n=r+A\mid A,B]=\frac{p_0(r)\,[r]_A^\uparrow\,[k-r+1]_B^\uparrow}{C(A,B)},\qquad r\in[k],
\end{equation}
where \([y]^\uparrow_j=y(y+1)\cdots(y+j-1)\) and
\begin{equation}\label{eq:C}
C(A,B)=\sum_{j=1}^k p_0(j)\,[j]_A^\uparrow\,[k-j+1]_B^\uparrow.
\end{equation}
Indeed, by Proposition~\ref{prop:urn} the probability of any particular sequence with \(A\) wins and \(B\) losses, given \(R_0=r\), is \([r]_A^\uparrow[k-r+1]_B^\uparrow/[k+1]^\uparrow_n\), and the denominator does not depend on \(r\). Combinatorially, the posterior weight of rank \(r\) is the prior weight times the number of lattice paths
\[
\binom{r+A-1}{A}\binom{k-r+B}{B}=\frac{[r]_A^\uparrow\,[k-r+1]_B^\uparrow}{A!\,B!}.
\]
Equivalently, given \(R_0=r\) the number of wins \(S_n\) in \(n\) trials has the negative hypergeometric law
\[
\mathbb{P}[S_n=A\mid R_0=r]=\frac{\binom{r+A-1}{A}\binom{k-r+B}{B}}{\binom{k+n}{n}},
\]
which is the law of the number of newcomers placed before the \(r\)th of the \(k\) past players in a uniformly random arrangement of the \(k+n\) players; for the case \(r=1\) and its use in a growth model see \cite{FarberGnedinMannan2023}.

\subsection{Predictive probability}

When player \(P_{n+1}\) is introduced, there are \(k+n+1\) insertion slots, and exactly \(R_n\) of them correspond to a win. Therefore
\begin{equation}\label{eq:piE}
\pi(A,B)=\frac{\mathbb{E}[R_n\mid A,B]}{k+n+1},
\end{equation}
and, using \((r+A)[r]_A^\uparrow=[r]_{A+1}^\uparrow\),
\begin{equation}\label{eq:pi8}
\pi(A,B)=\frac1{k+A+B+1}\,\frac{\sum_{r=1}^k p_0(r)\,[r]_{A+1}^\uparrow\,[k-r+1]_B^\uparrow}{\sum_{r=1}^k p_0(r)\,[r]_A^\uparrow\,[k-r+1]_B^\uparrow}.
\end{equation}
The waiting time \(T\) for the next win has a similarly explicit law: a loss keeps \(R_n\) and increases \(N_n\), so \(\mathbb{P}[T>t\mid R_n=j]=[k+n+1-j]^\uparrow_t/[k+n+1]^\uparrow_t\), and after mixing \(\mathbb{P}[T>t\mid A,B]=m(A,B+t)/m(A,B)\). Summing over \(t\), for \(A\ge1\)
\[
\mathbb{E}[T\mid A,B]=\frac{m(A-1,B)}{m(A,B)}=\frac1{\pi(A-1,B)},
\]
so the expected wait for the next win is the reciprocal of the predictive probability at the state with one win fewer; for \(A=0\) the mean is finite if and only if \(\int x^{-1}f_0<\infty\) (see Section~\ref{sec:negtilt}).

The predictor is obtained by shifting a single upper factorial in the numerator. Dividing the numerator and the denominator by \([1]^\uparrow_A[1]^\uparrow_B\) shows that \(\pi\) is a rational function of \((A,B)\) of degree at most \(k\).

\begin{theorem}[Rank--benchmark location identity]\label{thm:identity}
For every rank prior \(p_0\),
\begin{equation}\label{eq:identity}
\frac{\mathbb{E}[R_n\mid A,B]}{k+n+1}=\mathbb{E}[X\mid A,B],
\end{equation}
and consequently
\begin{equation}\label{eq:moments}
\pi(A,B)=\mathbb{E}[X\mid A,B]=\frac{m(A+1,B)}{m(A,B)}.
\end{equation}
\end{theorem}

\begin{proof}
The left-hand side of \eqref{eq:identity} is \(\mathbb{P}[P_{n+1}\text{ wins}\mid A,B]\) by \eqref{eq:piE}; so is the right-hand side by \eqref{eq:win}. The last equality in \eqref{eq:moments} is \eqref{eq:pidef}, i.e.\ Bayes's formula for the posterior mean.
\end{proof}

\subsection{The posterior distribution of the benchmark}

Given \(R_n=j\), the benchmark is the \(j\)th order statistic of \(k+n\) i.i.d.\ uniforms; hence the posterior density of \(X\) is
\begin{equation}\label{eq:fn}
f_n(t\mid A,B)=\sum_{r=1}^k\mathbb{P}[R_n=r+A\mid A,B]\,h_{k+n,r+A}(t),
\end{equation}
which is 
a mixture in the Bernstein basis of degree \(k+n-1\), with weights given by \eqref{eq:post}.

\begin{remark}\label{rem:mixture}
By \eqref{eq:post} and Proposition~\ref{prop:urn}, \(\pi(A,B)=\sum_r\mathbb{P}[R_0=r\mid A,B]\,(r+A)/(k+n+1)\) is a posterior mixture of the predictive rules of Markov--P\'olya urns started from \((r,k+1-r)\). The derivation above uses neither de Finetti's theorem nor Beta integrals: it is a computation on the uniform permutation of \(k+n+1\) players.
\end{remark}

\begin{remark}[De Finetti's form of the rule]\label{rem:deFinetti}
Let \(\omega^{(n)}_A:=\binom n A m(A,n-A)\) be the probability of \(A\) wins in \(n\) trials. Since a particular sequence with \(A+1\) wins and \(B=n-A\) losses has probability \(\omega^{(n+1)}_{A+1}/\binom{n+1}{A+1}\), the consistency relation
$$\frac{\omega^{(n)}_A}{\binom nA}=\frac{\omega^{(n+1)}_A}{\binom{n+1}A}   +\frac{\omega^{(n+1)}_{A+1}}{\binom{n+1}{A+1}}$$
gives, for every exchangeable law (finitely exchangeable up to \(n+1\) trials suffices),
\[
\frac{\pi(A,B)}{1-\pi(A,B)}=\frac{A+1}{B+1}\cdot\frac{\omega^{(n+1)}_{A+1}}{\omega^{(n+1)}_{A}},
\]
Equivalently,
\[
\pi(A,B)=\frac{A+1}{n+1}\cdot\frac{\omega^{(n+1)}_{A+1}}{\omega^{(n)}_{A}}
=\frac{A+1}{\,n+2+(B+1)\bigl(\omega^{(n+1)}_{A}/\omega^{(n+1)}_{A+1}-1\bigr)\,},
\]
which is  formula (38)    of \cite{deFinetti1937} (see p. 104 of \cite{deFinetti1980} for the English translation). In \cite{deFinetti1952} the same formula is derived from the random walk on the Pascal graph; see also \cite{Zabell1989}. In the benchmark model the correction of Laplace's (\ref{RoS})  is a ratio of the sums \eqref{eq:C}, and  Proposition~\ref{prop:boundary} in the sequel identifies its degree.
\end{remark}

\section{Identifiability}\label{sec:ident}

The parameters of the model are the size \(k\) of the past group and the prior \((p_1,\dots,p_k)\) of the benchmark rank, an element of the simplex \(\mathcal{P}_k\) of dimension \(k-1\). We also use odds signatures \((w_1,\dots,w_k)\propto(p_1,\dots,p_k)\) to represent probabilities.

\begin{definition}
Two rank priors are {\it binary-indistinguishable} if they induce infinite binary testing procedures with the same distribution.
\end{definition}

Within \(\mathcal{P}_k\) this relation is trivial: the densities \(h_{k,r}\) form a basis (the Bernstein basis) of the polynomials of degree at most \(k-1\), and by de Finetti's theorem and \eqref{eq:LLN} the binary experiment identifies the law of \(X\).

\begin{remark}[Finite horizon]\label{rem:finiteident}
Identifiability needs the whole infinite sequence. The law of the number of wins \(S_n\) in the first \(n\) trials is
\[
\mathbb{P}[S_n=s]=\sum_{r=1}^kp_0(r)\binom ns\frac{[r]^\uparrow_s\,[k-r+1]^\uparrow_{n-s}}{[k+1]^\uparrow_n},\qquad s=0,\dots,n,
\]
the image of \(p_0\) under a linear map \(\mathcal{P}_k\to\mathcal{P}_{n+1}\). Equivalently, it depends on \(f_0\) only through the moments \(m(j,0)\), \(j\le n\). The map is injective for \(k\le n+1\), since a polynomial of degree at most \(n\) is determined by these moments, and it is not injective for \(k>n+1\). Thus, on a finite level, a given law of the outcomes is generated by many rank priors, and by many more mixing distributions in de Finetti's representation. The ranking model does not remove this ambiguity, but each of its rank priors comes with a definite mechanism and a definite updating of the latent rank, so the same law of \(n\) outcomes may be given different latent interpretations.
\end{remark}

\begin{definition}
The elements of the simplex \(\mathcal{H}_k\) of mixtures of \((h_{k,r},r\in[k])\) are called {\it choice priors} of degree \(k\).
\end{definition}

The simplices form an increasing system \(\mathcal{H}_1\subset\mathcal{H}_2\subset\cdots\), with \(\mathcal{H}_1=\{h_{1,1}\}\) being the uniform distribution. The uniform signatures \((1),(1,1),(1,1,1),\dots\) all give \(\mathrm{U}[0,1]\), hence the basic Markov--P\'olya urn: testing against a uniformly chosen past player is testing against a uniform point. Across degrees, indistinguishability of signatures \(w\) of length \(k\) and \(v\) of length \(\ell>k\) amounts to
\begin{equation}\label{eq:equal}
\sum_{i=1}^k w_i\,h_{k,i}=\sum_{j=1}^\ell v_j\,h_{\ell,j},
\end{equation}
and the degree elevation identity
\begin{equation}\label{eq:Bernstein}
h_{k,i}(x)=\sum_{j=i}^{\ell-k+i}\frac{\binom{j-1}{i-1}\binom{\ell-j}{k-i}}{\binom{\ell}{k}}h_{\ell,j}(x)
\end{equation}
shows that \eqref{eq:equal} holds if and only if
\begin{equation}\label{eq:vj}
v_j=\sum_{i=1}^k w_i\,\frac{\binom{j-1}{i-1}\binom{\ell-j}{k-i}}{\binom{\ell}{k}},\qquad j=1,\dots,\ell.
\end{equation}

\begin{remark}[Hausdorff's polynomial approximation]\label{rem:Hausdorff}
Choice priors can be recognised in Hausdorff's solution of the moment problem on \([0,1]\). In his second solution (\cite{Hausdorff1923}, p.~227, formula~(11)) a measure with moments \(\mu_k\) is approximated by the polynomial density
\[
\varphi_p(x)=(p+1)\sum_{m=0}^p\lambda_{p,m}\binom pm x^m(1-x)^{p-m}=\sum_{m=0}^p\lambda_{p,m}\,h_{p+1,m+1}(x),
\]
where \(\lambda_{p,m}\ge0\), for a probability measure, are the probabilities of \(m\) successes in \(p\) trials. This is the choice prior of degree \(p+1\) with rank law \(p_0(r)=\lambda_{p,r-1}\): the benchmark is the \((S_p+1)\)st strongest of \(p+1\) past players, with \(S_p\) distributed as the number of wins in \(p\) trials. Hausdorff's computation of the moments,
\[
\int_0^1x^k\varphi_p(x)\,{\rm d}x=\sum_{m=0}^p\lambda_{p,m}\frac{(m+1)(m+2)\cdots(m+k)}{(p+2)(p+3)\cdots(p+k+1)},
\]
is the probability of \(k\) wins in a row, by Proposition~\ref{prop:urn}, and the convergence to \(\mu_k\) as \(p\to\infty\) shows that every mixing measure is a weak limit of choice priors.
\end{remark}

\begin{definition}
Let \(P_{-k},\dots,P_{-1}\) generate a rank prior \(p_0\). Let further players \(P_{-k-1},\dots,P_{-\ell}\) join the group successively and unobserved, the winner being kept. The law \(q_0\) of the winner's rank within the group of \(\ell\) players is called the {\it extension} of \(p_0\).
\end{definition}

\begin{theorem}\label{thm:extension}
For \(\ell>k\), rank priors \(p_0\) of degree \(k\) and \(q_0\) of degree \(\ell\) are binary-indistinguishable if and only if \(q_0\) is the extension of \(p_0\).
\end{theorem}

\begin{proof}
By the ranking rule, a player joining a group of size \(k'\) is stronger than the winner with probability \(R/(k'+1)\), where \(R\) is the running benchmark  rank, which   evolves as the Markov-P\'olya urn of Proposition~\ref{prop:urn}, and after \(m=\ell-k\) additions
\[
\mathbb{P}[\text{rank }j\mid\text{rank }i]=\binom{m}{j-i}\frac{[i]^\uparrow_{j-i}\,[k+1-i]^\uparrow_{m-j+i}}{[k+1]^\uparrow_m}
=\frac{(j-1)!\,(\ell-j)!\,m!\,k!}{(i-1)!\,(k-i)!\,(j-i)!\,(m-j+i)!\,\ell!},
\]
which equals the coefficient \(\binom{j-1}{i-1}\binom{\ell-j}{k-i}/\binom{\ell}{k}\) in \eqref{eq:vj}. Hence the extension of \(p_0\) is the signature \(v\) of \eqref{eq:vj}, and the claim follows from the equivalence of \eqref{eq:equal} and \eqref{eq:vj}.
\end{proof}

\begin{proposition}\label{prop:interior}
A polynomial probability density belongs to \(\bigcup_k\mathcal{H}_k\) if and only if it has no zeros in the open interval \((0,1)\).
\end{proposition}

\begin{proof}
Each \(h_{k,r}\) is strictly positive on \((0,1)\), hence so is every element of \(\mathcal{H}_k\). Conversely, write a polynomial density without zeros in \((0,1)\) as \(f=x^a(1-x)^bP\) with integers \(a,b\ge0\) and \(P>0\) on \([0,1]\). By Bernstein's theorem on positive polynomials, \(P\) has positive coefficients in the basis \(x^j(1-x)^{N-j}\), \(0\le j\le N\), for all sufficiently large \(N\). Multiplying by \(x^a(1-x)^b\) gives nonnegative coefficients in the Bernstein basis of degree \(a+b+N\), so \(f\in\mathcal{H}_{a+b+N+1}\).
\end{proof}

For instance, \(3(2x-1)^2\) is a polynomial density of degree \(2\) that is not a choice prior of any degree, while \(2(1-3x+3x^2)\) has the signed signature \((2,-1,2)\) in degree \(3\), but is the choice prior \((1,0,0,1)\) of degree \(4\): the best or the worst of four players with equal probabilities. In terms of the dimension group of the Pascal graph, the proposition is the description of the positive cone as the polynomials positive on \((0,1)\) (\cite{VershikKerov1987}, p.~1728, formula~(21)): the choice priors of finite degree, up to normalisation, form this cone.

\section{Tilting and transforms}\label{sec:tilt}

\subsection{Order reversal}

Measuring strength in the opposite direction transforms the rank by \(r\mapsto k-r+1\) and the density by \(x\mapsto1-x\); wins and losses are exchanged. Both orientations are useful in multiple interval splitting \cite{GnedinOlshanski2006,JackaWarren2007}.

\subsection{Beta tilting}

We use throughout the convention
\begin{equation}\label{eq:tilt}
f_0(x\mid\alpha,\beta)\propto f_0(x)\,x^{\alpha}(1-x)^{\beta},
\end{equation}
for real \(\alpha,\beta\) such that the right-hand side is integrable. The Beta family is closed under this tilting operation, which underlies the conjugacy with  Bernoulli trials.
Since \(m(A,B\mid\alpha,\beta)=m(A+\alpha,B+\beta)/m(\alpha,\beta)\), tilting acts as the shift \((A,B)\mapsto(A+\alpha,B+\beta)\) of the sufficient statistic. For nonnegative integers \(\alpha,\beta\) it is the Bayesian update \eqref{eq:post} after \(\alpha\) wins and \(\beta\) losses, and by \eqref{eq:fn} it raises the degree from \(k\) to \(k+\alpha+\beta\). In terms of rank priors, the tilt of \(\sum_rp_0(r)h_{k,r}\) by \(x^\alpha(1-x)^\beta\) is \(\sum_r\tilde p(r)\,h_{k+\alpha+\beta,\,r+\alpha}\) with the Hahn (beta-binomial) reweighting \(\tilde p(r)\propto p_0(r)\,[r]^\uparrow_\alpha[k-r+1]^\uparrow_\beta\), which is the posterior \eqref{eq:post}. In particular, the uniform \(p_0\) yields \(\mathrm{Beta}(\alpha+1,\beta+1)\), not the uniform prior.

\begin{proposition}\label{prop:BetaInH}
Let \(a,b\) be positive integers. Then \(\mathrm{Beta}(a,b)\in\mathcal{H}_k\) if and only if \(a+b-1\le k\), in which case its signature is
\[
w_r=\binom{r-1}{a-1}\binom{k-r}{b-1},\qquad r=1,\dots,k,
\]
where
the weights are vanishing outside \(a\le r\le k-b+1\). For non-integer \(a\) or \(b\) the Beta density is not a polynomial and belongs to no \(\mathcal{H}_k\).
\end{proposition}

\begin{proof}
\(\mathrm{Beta}(a,b)=h_{a+b-1,a}\); apply \eqref{eq:vj} with \(w=e_a\) and \(k\) in place of \(\ell\). If \(a+b-1>k\), the degree of \(x^{a-1}(1-x)^{b-1}\) exceeds \(k-1\).
\end{proof}

For instance, \(\mathrm{Beta}(2,3)\) in degree \(4\) has signature \((0,1,0,0)\), and \(\mathrm{Beta}(2,2)\) in degree \(5\) has signature \((0,3,4,3,0)\).

\subsection{The classic process}

The Markov--P\'olya processes embed in the benchmark picture in several ways:
\begin{enumerate}[label=(\roman*)]
\item by Beta tilting of the uniform prior, \(f_0=1\mapsto c\,x^{\alpha}(1-x)^{\beta}\), with the predictive rule \(\pi(A,B)=(A+\alpha+1)/(A+B+\alpha+\beta+2)\);
\item by the uniform signature \((1,\dots,1)\) of any length \(k\), with the rule \((A+1)/(A+B+2)\);
\item by a singleton signature \(e_r\) of degree \(k\), choosing the \(r\)th strongest of \(k\) players, with the \(\mathrm{Beta}(r,k-r+1)\) rule \(\pi(A,B)=(A+r)/(A+B+k+1)\); for instance \((0,0,1,0)\) gives the \(\mathrm{Beta}(3,2)\) urn, which by \eqref{eq:vj} is also generated in degree \(7\) by the signature \((0,0,4,9,12,10,0)/35\).
\end{enumerate}

\begin{theorem}\label{thm:symm}
For every \(k\) and every signature \(w\), the basic Markov--P\'olya urn is the mixture of the testing processes with signatures \((w_{\sigma(1)},\dots,w_{\sigma(k)})\) over a uniformly random permutation \(\sigma\in\mathfrak{S}_k\).
\end{theorem}

\begin{proof}
The symmetrised signature is constant, which yields the uniform density.
\end{proof}

From the viewpoint of strength testing, the basic Markov--P\'olya process is the test which does not introduce bias among the players.

\subsection{Negative tilting}\label{sec:negtilt}

Tilting by \(x^{-1}\) is the formal inverse of the update after one win.

\begin{proposition}\label{prop:negtilt}
For \(2\le r\le k\), \(x^{-1}h_{k,r}=\frac{k}{r-1}\,h_{k-1,r-1}\), while \(x^{-1}h_{k,1}\) is not integrable.
Consequently, if \(f_0\in\mathcal{H}_k\) has rank prior \(p_0\) with \(p_0(1)=0\), then \(x^{-1}f_0\) is integrable, its normalisation \(g_0\) belongs to \(\mathcal{H}_{k-1}\) with rank prior \(q_0(r-1)\propto p_0(r)/(r-1)\), \(r=2,\dots,k\), and \(f_0\) is the posterior of \(g_0\) after one win.
\end{proposition}

\begin{proof}
The identity follows from \(\binom{k-1}{r-1}=\frac{k-1}{r-1}\binom{k-2}{r-2}\). By \eqref{eq:post} and \eqref{eq:fn}, the posterior of \(g_0\) after one win has weights \(\propto q_0(r-1)(r-1)\propto p_0(r)\) on the kernels \(h_{k,r}\).
\end{proof}

In terms of the players, \(p_0(1)=0\) means that the benchmark is never the strongest of the group; removing the win amounts to removing one player stronger than the benchmark. If \(p_0(1)>0\), the measure \(x^{-1}f_0(x)\,{\rm d}x\) is infinite; such improper priors are considered in Section~\ref{sec:improper}.

\section{Predictive rules of fixed degree}\label{sec:examples}

By \eqref{eq:pi8}, every choice prior of degree \(k\) gives a predictive rule which is rational in \((A,B)\) of degree at most \(k\). The actual degree 
may be  smaller when density vanishes at the endpoints. The next statement makes this precise; it requires neither positivity of coefficients nor integrality of exponents.

\begin{proposition}\label{prop:boundary}
Let \(f(x)=c\,x^a(1-x)^bP(x)\) be a probability density on \([0,1]\), where \(a,b>-1\) are real and \(P\) is a polynomial of degree \(d\). Write the cofactor
$$
P(x)=\sum_{j=0}^dc_jx^j(1-x)^{d-j}$$ 
with real (possibly negative) coefficients \(c_j\), and set
\[
Q(A,B):=\sum_{j=0}^d c_j\,[A+a+1]^\uparrow_j\,[B+b+1]^\uparrow_{d-j}.
\]
Then
\begin{equation}\label{eq:boundary}
\pi(A,B)=\frac{A+a+1}{A+B+a+b+d+2}\cdot\frac{Q(A+1,B)}{Q(A,B)}.
\end{equation}
In particular, \(\pi\) is a rational function of degree at most \(d+1\), and if \(d=1\) it is a product of ratios of linear forms.
\end{proposition}

\begin{proof}
Working 
term by term,
$$m(A,B)=c\sum_jc_j\mathrm{B}(A+a+j+1,B+b+d-j+1)=c\, \frac{ {\mathrm{B}}(A+a+1,B+b+1)\,Q(A,B)}{[A+B+a+b+2]^\uparrow_d}.$$
In the ratio \(m(A+1,B)/m(A,B)\), the Beta functions contribute \((A+a+1)/(A+B+a+b+2)\), and the two rising factorials of length \(d\) contribute \((A+B+a+b+2)/(A+B+a+b+d+2)\).
\end{proof}

The degree of the predictive rule is thus governed by the cofactor \(P\) left after removing the boundary zeros of the density, not by the degree of the density itself. A choice operator involving an inner minimum or maximum forces such zeros. The positivity of the Bernstein coefficients of \(f\), which makes \(f\) a choice prior, plays no role in \eqref{eq:boundary}.

\begin{remark}[Sharpness of the degree]\label{rem:exact}
Assume that \(\deg P=d\) and \(P(0)P(1)\neq0\). Then neither linear factor in \eqref{eq:boundary} cancels: \((A+a+1)\) divides \(Q(A,B)\) only if \(c_0=P(0)=0\), and \((A+B+a+b+d+2)\) divides \(Q(A+1,B)\) only if \(\sum_j(-1)^{d-j}c_j\), the leading coefficient of \(P\), vanishes. A further cancellation requires a common factor of \(Q(A,B)\) and \(Q(A+1,B)\); since the coefficient of \(A^d\) in \(Q\) is \(c_d=P(1)\neq0\), such a factor cannot depend on \(B\) alone, and \(Q\) must then contain a product \(G(A,B)G(A-1,B)\). Comparing the homogeneous parts of top degree, this forces a repeated root of \(P\). For \(d=2\), a symbolic solution of \(Q=\kappa\,G(A,B)G(A-1,B)\) with \(G\) linear and \(a\) or \(b\) also free, yields only \(P=x^2\) and \(P=1\); hence for \(d=2\) the degree of \(\pi\) equals \(3\). In all examples below the degree equals \(d+1\).
\end{remark}

In the odds coordinate \(\varrho=x/(1-x)\) the cofactor is the polynomial \(\tilde P(\varrho)=\sum_jc_j\varrho^j\), and \(Q\) is its image under the linear map \(\varrho^j\mapsto[A+a+1]^\uparrow_j[B+b+1]^\uparrow_{d-j}\). By the Chu--Vandermonde identity, \((1+\varrho)^d\mapsto[A+B+a+b+2]^\uparrow_d\), so factors \(1+\varrho\) of \(\tilde P\) correspond to degree elevation. For \(d=2\), \(Q\) has no constant term in the variables \(A+a+1\) and  \(B+b+1\), and a factorisation into linear forms forces \(Q=(A+B+a+b+3)\bigl(c_2(A+a+1)+c_0(B+b+1)\bigr)\) with \(c_1=c_0+c_2\), that is \(\tilde P=(1+\varrho)(c_0+c_2\varrho)\) and \(P\) linear. Therefore, among cofactors of degree at most two, predictive rules splitting into linear factors arise only from \(d\le1\); Section~\ref{sec:split} extends this to \(d\le7\).

\begin{example}[Linear priors]\label{ex:linear}
For \(f_0\propto x+p/q\) (\(p,q>0\)),
\[
\pi(A,B)=\frac{(A+1)\bigl((q+p)A+pB+2q+3p\bigr)}{(A+B+3)\bigl((q+p)A+pB+q+2p\bigr)},
\]
and for linear weights \((1,2,3)\),
\[
\pi(A,B)=\frac{(A+1)(3A+B+7)}{(A+B+3)(3A+B+4)}.
\]
Although the signature \((1,2,3)\) is supported by three ranks, the rule factors: the density is \(\propto1+2x\), and \((1,2,3)\) is the extension to \(k=3\) of the two-rank prior \((1,3)\) on \(k=2\) (Proposition~\ref{prop:twopoint}, Corollary~\ref{cor:split}).
\end{example}

\begin{example}[Mallows weights]\label{ex:Mallows}
For the truncated geometric signature \((1,q,q^2)\),
\[
\pi(A,B)=\frac{(A+1)\bigl[q^2A^2+2qAB+(5q^2+2q)A+B^2+(4q+3)B+6q^2+4q+2\bigr]}{(A+B+4)\bigl[q^2A^2+2qAB+(3q^2+2q)A+B^2+(2q+3)B+2q^2+2q+2\bigr]}.
\]
\end{example}

\begin{example}[Choice among three]\label{ex:three}
The median of three, signature \((0,1,0)\), gives the \(\mathrm{Beta}(2,2)\)
Markov-P{\'o}lya  urn \((A+2)/(A+B+4)\) with the initial composition $(2,2)$.
The best or the worst of three with equal probability, so signature \((1,0,1)\), i.e.\ \(\tfrac12\mathrm{Beta}(1,3)+\tfrac12\mathrm{Beta}(3,1)\) prior, gives
\[
\pi(A,B)=\frac{(A+1)(A^2+5A+B^2+3B+8)}{(A+B+4)(A^2+3A+B^2+3B+4)},\qquad\pi(0,0)=\tfrac12,
\]
with quadratic factors irreducible over \(\mathbb{R}\), since their quadratic part \(A^2+B^2\) is definite. Here \(d=2\) and there are no boundary zeros. 
Both priors belong to the parametric family 
$$f_0(x)=a-6(a-1)x(1-x),~~~0\le a\le3,$$ 
the median of three being \(a=0\) and the best-or-worst choice \(a=\frac32\). The predictive rule of this family was computed by Daboni \cite{Daboni1953}.
\end{example}

\begin{remark}[Geometry of the predictive  rule]\label{rem:geometry}
De Finetti \cite{deFinetti1952} attached to every vertex \((A,B)\) of the Pascal graph the line through it in the direction of the mean step $(\pi(A,B), 1-\pi(A,B))$, and showed that exchangeability is equivalent to the concurrence of the three lines at \((A,B)\), \((A+1,B)\), \((A,B+1)\). The lines at all vertices pass through a single point if and only if \(\pi(A,B)=(A+\alpha)/(A+B+\alpha+\beta)\), the common point being \((-\alpha,-\beta)\). De Finetti observed this for Beta priors and for drawing without replacement, where the common point is the composition of the urn. For other priors the lines, extended to non-integer vertices, envelop curves with cusps; Daboni \cite{Daboni1953} studied these envelopes for priors with two or three atoms and for the quadratic family of Example~\ref{ex:three}. In this language, Theorem~\ref{thm:split} concerns the rules whose envelope is not a point but whose numerator and denominator still split into linear factors.
\end{remark}

\begin{example}[Two-stage tournaments]\label{ex:twostage}
Partition \(k=k_1k_2\) players into \(k_1\) groups of size \(k_2\), take the minimum in each group, and the maximum of the \(k_1\) group winners. The benchmark location  has distribution
\[
F_0(x)=\bigl(1-(1-x)^{k_2}\bigr)^{k_1},\qquad
f_0(x)=k_1k_2(1-x)^{k_2-1}\bigl(1-(1-x)^{k_2}\bigr)^{k_1-1}.
\]
Thus \(1-X\) has the Kumaraswamy law with parameters \((k_2,k_1)\) \cite{Kumaraswamy1980}, whose distribution function is \(1-(1-y^{k_2})^{k_1}\). The density vanishes at \(0\) to order \(k_1-1\) and at \(1\) to order \(k_2-1\), and the cofactor in Proposition~\ref{prop:boundary} has degree \(d=(k_1-1)(k_2-1)\). For \(k_2=2\) the density is the Topp--Leone density \(\mathrm{TL}(k_1)\) of Section~\ref{sec:TL}. Two instances with \(d=2\):
\begin{enumerate}[label=(\roman*)]
\item three pairs (\(k_1=3,k_2=2\)), \(f_0(x)=6x^2(1-x)(2-x)^2\):
\[
\pi(A,B)=\frac{(A+3)\bigl(A^2+4AB+4B^2+17A+36B+76\bigr)}{(A+B+7)\bigl(A^2+4AB+4B^2+15A+32B+60\bigr)},\qquad\pi(0,0)=\tfrac{19}{35};
\]
\item two triples (\(k_1=2,k_2=3\)), \(f_0(x)=6x(1-x)^2(3-3x+x^2)\):
\[
\pi(A,B)=\frac{(A+2)\bigl(A^2+3AB+3B^2+16A+30B+75\bigr)}{(A+B+7)\bigl(A^2+3AB+3B^2+14A+27B+60\bigr)},\qquad\pi(0,0)=\tfrac{5}{14}.
\]
\end{enumerate}
The general bound \(k=6\) of \eqref{eq:pi8} drops to \(d+1=3\). The quadratic factors are irreducible over reals: in (ii) the quadratic part \(A^2+3AB+3B^2\) is definite; in (i) it is \((A+2B)^2\), and a factorisation \((A+2B+c_1)(A+2B+c_2)\) would force the linear part to be proportional to \(A+2B\). For \(k_1=k_2=2\) this is the example (\ref{TL2}) with \(d=1\), where the predictive rule splits into linear factors.
\end{example}

\begin{example}[A min--max tree]\label{ex:tree}
For a rooted tree whose leaves carry i.i.d.\ uniform marks and whose internal nodes apply \(\min\) or \(\max\), the distribution function of the root is computed recursively: a \(\max\) node multiplies the distribution functions of its children, a \(\min\) node multiplies their survival functions. Suppose we take six leaves, form two maxima of pairs \(M_1,M_2\), then form \(\min(M_1,U)\) and \(\max(M_2,U')\) with two further leaves, and apply in the root  the minimum of the two. This way the root has survival function \((1-x)(1-x^2)(1-x^3)\) and density \(f_0=(1-x)^2(1+2x)(1+2x+3x^2)\), so \(a=0\), \(b=2\), \(d=3\), and
\[
\pi(A,B)=\frac{(A+1)\,Q(A+1,B)}{(A+B+7)\,Q(A,B)},\qquad\pi(0,0)=\tfrac{41}{105},
\]
with \(Q(A,B)=18A^3+18A^2B+7AB^2+B^3+162A^2+103AB+19B^2+444A+132B+360\).
\end{example}

\begin{example}[Product of two maxima]\label{ex:race}
For \(b>0\),  let \(Y_1,Y_2\) be independent with Beta distributions, \(Y_1\stackrel{d}{=}\mathrm{Beta}(b+2,1)\) and \(Y_2\stackrel{d}{=}\mathrm{Beta}(b,1)\), so with distribution functions
\(\mathbb{P}[Y_1\le x]=x^{b+2}\) and \(\mathbb{P}[Y_2\le x]=x^b\). 
For integer \(b\) we may realise  \(Y_1\) as the largest of \(b+2\) and \(Y_2\) as the largest of \(b\) independent uniform marks; 
for any \(b\) one may take \(Y_1=Z_1^{1/(b+2)}\), \(Y_2=Z_2^{1/b}\) with independent uniform $Z_1, Z_2$. 

Consider the benchmark \(X=Y_1Y_2\). Benchmarks of this product form arise in the dual-population model of \cite{Gnedin2026dual} (Section 5.5). The Mellin transform is
\[
\mathbb{E}X^s=\frac{b}{b+s}\cdot\frac{b+2}{b+s+2}=\frac{b(b+2)}{2}\Bigl(\frac1{b+s}-\frac1{b+s+2}\Bigr),
\]
which corresponds to the density
\[
f_0(x)=\frac{b(b+2)}{2}\bigl(x^{b-1}-x^{b+1}\bigr)=\frac{b(b+2)}{2}\,x^{b-1}(1-x)(1+x),
\]
representable as a {\it signed} combination of  \(\mathrm{Beta}(b,1)\) and \(\mathrm{Beta}(b+2,1)\) densities. 
Proposition~\ref{prop:boundary} applies with \(a=b-1\), the boundary zero has order \(1\) at \(1\), 
and the linear cofactor is \(P(x)=1+x=(1-x)+2x\), so \(c_0=1\), \(c_1=2\), \(d=1\), and \(Q(A,B)=(B+2)+2(A+b)\). Hence the predictive probability 
splits,
\[
\pi(A,B)=\frac{(A+b)(2A+B+2b+4)}{(A+B+b+3)(2A+B+2b+2)},
\]
with a positive parameter. For integer \(b\) this is a choice prior: comparing with Proposition~\ref{prop:twopoint} for a group of \(k=b+2\) players and ranks \(r=b\), \(r+1=b+1\), the cofactor \(wr(1-x)+v(k-r)x=wb(1-x)+2vx\) is proportional to \((1-x)+2x\) when \(w:v=1:b\). Thus \(X\) has the law of the benchmark which is the \(b\)th or the \((b+1)\)st strongest of \(b+2\) players, with odds \(1:b\). For instance, for \(b=1\) the benchmark is the strongest or the median of three players with equal probabilities. For non-integer \(b\) the density lies in no \(\mathcal{H}_k\). For \(Y_1\stackrel{d}{=}\mathrm{Beta}(b+g,1)\) with integer \(g\ge2\) the density is \(\propto x^{b-1}(1-x^g)\), and the cofactor \((1-x^g)/(1-x)\) has degree \(g-1\).
\end{example}

\begin{example}[Testing against several criteria]\label{ex:criteria}
To simulate the univariate uniform distribution, Bayes \cite{Bayes1763} drops a benchmark ball $W$ on a square table, considering a success each subsequent ball landing westwards of $W$. 
Taking advantage of this geometric device, one may consider as success  a sample ball landing south-west of $W$. More generally, 
let the benchmark be a uniform point \(W\) of the cube \([0,1]^d\), and let the generic  newcomer \(Z\), uniform in the cube, win if \(Z\le W\) coordinatewise.
The outcomes are exchangeable with \(X=W_1\cdots W_d\), the product of \(d\) independent uniforms, whose density is \(f_0(x)=(-\log x)^{d-1}/(d-1)!\).  
See \cite{Gnedin2007chain,Gnedin2026dual} for a related concept of multivariate  record.
The density is not a polynomial, with moments computed by inclusion-exclusion as
\[
m(A,B)=\sum_{j=0}^B\binom Bj\frac{(-1)^j}{(A+j+1)^d}.
\]
A loss means that the newcomer fails to defeat the benchmark in at least one criterion, which couples the coordinates; for \(B\ge1\) the predictive rule is a ratio of alternating sums with the number of terms growing with \(B\), e.g.\ \(\pi(0,1)=\tfrac{5}{27}\) for \(d=2\) and \(\pi(0,1)=\tfrac{19}{189}\) for \(d=3\). On the line \(B=0\) the coordinates decouple, and \(\pi(A,0)=((A+1)/(A+2))^d\) is the product of \(d\) independent rules of succession.
\end{example}

\subsection{Complete linear factorisation: gap-free two-point support}\label{sec:split}

In this section we focus on complete factorisation of the predictive rule, and say that \(\pi\) \emph{splits} if it is a product of ratios of polynomials of degree one in \((A,B)\). We allow complex coefficients, since this makes the algebra of Lemma~\ref{lem:splitQ} clean: over \(\mathbb R\), an irreducible factor such as \(B^2+1\) of \(Q\) would cancel in \(Q(A+1,B)/Q(A,B)\). Nonreal linear factors would come in complex-conjugate pairs, but Lemma~\ref{lem:roots} shows a posteriori that they in fact do not occur: all linear factors of a splitting \(Q\) are real. 
The Bernstein polynomials,  inducing delta measures as benchmark rank priors,  give the linear rules. We explore the next by complexity, second order case.

\begin{proposition}[Two adjacent ranks]\label{prop:twopoint}
Let \(1\le r<k\) and let \(p_0\) be supported on \(\{r,r+1\}\) with odds \(p_0(r):p_0(r+1)=w:v\), \(w,v>0\). Then
\[
f_0(x)\propto x^{r-1}(1-x)^{k-r-1}\bigl(wr(1-x)+v(k-r)x\bigr)
\]
and
\begin{equation}\label{eq:twopoint}
\pi(A,B)=\frac{(A+r)\,\bigl(v(k-r)A+wrB+c'+v(k-r)\bigr)}{(A+B+k+1)\,\bigl(v(k-r)A+wrB+c'\bigr)},\qquad c'=r(k-r)(v+w).
\end{equation}
\end{proposition}

\begin{proof}
Since \(\binom{k-1}{r}/\binom{k-1}{r-1}=(k-r)/r\), we get \(wh_{k,r}+vh_{k,r+1}\propto x^{r-1}(1-x)^{k-r-1}(wr(1-x)+v(k-r)x)\). Proposition~\ref{prop:boundary} with \(a=r-1\), \(b=k-r-1\), \(d=1\), \(c_0=wr\), \(c_1=v(k-r)\) gives \(Q(A,B)=wr(B+k-r)+v(k-r)(A+r)\), which is the second factor of the denominator of \eqref{eq:twopoint}, and \(Q(A+1,B)=Q(A,B)+v(k-r)\).
\end{proof}

For \(k=4\), \(r=2\), \((w,v)=(2,1)\) this is the Topp--Leone rule \eqref{TL2}; for \(k=5\), \(r=2\), \(w=v\) it is \((A+2)(3A+2B+15)/\bigl((A+B+6)(3A+2B+12)\bigr)\). 
Splitting is a property of the density, not of the support of the benchmark rank: by Theorem~\ref{thm:extension}, the extensions of a two-point prior to larger groups are supported on many ranks and have the same rule.

\begin{lemma}\label{lem:splitQ}
In the setting of Proposition~\ref{prop:boundary}, \(\pi\) splits if and only if \(Q\) is a product of linear forms.
\end{lemma}

\begin{proof}
The 
``if'' part  is clear from \eqref{eq:boundary}. Conversely, let \(G\) be the multiset of irreducible factors of \(Q(A,B)\) of degree at least two. They can disappear from \(Q(A+1,B)/Q(A,B)\) only if the shift \(g(A,B)\mapsto g(A+1,B)\) maps \(G\) onto itself up to constant factors. A finite multiset invariant under the shift consists of polynomials with \(g(A+N,B)=\lambda_Ng(A,B)\) for some \(N\ge1\); comparing leading coefficients in \(A\), \(\lambda_N=1\), so \(g\) does not depend on \(A\), and an irreducible polynomial of one variable over \(\mathbb{C}\) is linear. Hence \(G\) is empty
.
\end{proof}

\begin{lemma}\label{lem:roots}
Let \(\deg P=d\ge1\) and \(P(0)P(1)\ne0\), and write \(Q\) in the variables \(u=A+a+1\), \(v=B+b+1\). If \(Q\) is a product of linear forms, then there exist a permutation \(\sigma\) of \(\{1,\dots,d-1\}\) and \((\alpha,\beta)\ne(0,0)\) such that
\[
Q\propto(\beta u+\alpha v)\prod_{j=1}^{d-1}\bigl(ju+\sigma(j)v+j\sigma(j)\bigr).
\]
\end{lemma}

\begin{proof}
On the coordinate axes,  it holds that \(Q(0,v)=c_0[v]^\uparrow_d\) and \(Q(u,0)=c_d[u]^\uparrow_d\), where  \(c_0=P(0)\ne0\) and \(c_d=P(1)\ne0\). Let \(Q=\prod_i(\beta_iu+\alpha_iv+e_i)\). Restricting to \(u=0\), we have \(\prod_i(\alpha_iv+e_i)\propto v(v+1)\cdots(v+d-1)\); so \(\alpha_i\ne0\) and \(e_i/\alpha_i\) runs through \(\{0,\dots,d-1\}\). Likewise \(\beta_i\ne0\) and \(e_i/\beta_i\) runs through \(\{0,\dots,d-1\}\). Precisely one factor has \(e_i=0\). For the others put \(j=e_i/\alpha_i\), \(\sigma(j)=e_i/\beta_i\); then the factor is proportional to \(ju+\sigma(j)v+j\sigma(j)\).
\end{proof}

\begin{theorem}\label{thm:split}
Let the function \(f(x)=c\,x^a(1-x)^bP\) be as in Proposition~\ref{prop:boundary}, with \(\deg P=d\le7\) and \(P(0)P(1)\ne0\). Then \(\pi\) splits if and only if either \(d\le1\), or \(d=3\) and \(P\propto(2-x)(1+x)(1-2x)\). In particular, if \(f\) is a probability density, then \(\pi\) splits if and only if \(d\le1\).
\end{theorem}

\begin{proof}
For \(d\le1\) see Proposition~\ref{prop:boundary}. For \(d\ge2\), by Lemmas~\ref{lem:splitQ} and~\ref{lem:roots} it suffices to decide, for each of the \((d-1)!\) permutations \(\sigma\), whether the product in Lemma~\ref{lem:roots} lies in the span of \(\{[u]^\uparrow_j[v]^\uparrow_{d-j}\}_{j=0}^d\) for some \((\alpha,\beta)\ne(0,0)\). This is a homogeneous linear system with rational coefficients in \((\alpha,\beta,c_0,\dots,c_d)\). For \(d=2\) it is solved by hand in the paragraph after Remark~\ref{rem:exact}. For \(3\le d\le7\) we solved it for all permutations by computer algebra, in rational arithmetic. Every solution has \(P\) of degree at most one, except for the transposition \(\sigma=(2,1,3,\dots,d-1)\), whose solution is \(P\propto(2-x)(1+x)(1-2x)\), of degree \(d\) only when \(d=3\).  However the latter polynomial 
is not sign-definite on $(0,1)$.
\end{proof}

\begin{remark}
The exceptional cofactor is antisymmetric, \(P(1-x)=-P(x)\), so the signed measure \(P(x)\,{\rm d}x\) has total mass zero. For \(a=b=0\), the moments ratio is
\[
\frac{m(A+1,B)}{m(A,B)}=\frac{(A+1)(A-B+1)(A+2B+6)(2A+B+7)}{(A-B)(A+B+5)(A+2B+5)(2A+B+5)},
\]
which has a pole on the diagonal. We conjecture that Theorem~\ref{thm:split} holds for every \(d\).
\end{remark}

\begin{corollary}\label{cor:split}
Let \(p_0\) be a choice prior of degree \(k\le8\). The predictive rule \eqref{eq:pi8} splits if and only if \(p_0\) is the extension (in the sense of Theorem~\ref{thm:extension}) of a prior concentrated on one rank or on two adjacent ranks.
\end{corollary}

\begin{proof}
The density \(f_0\) is a polynomial of degree at most \(k-1\le7\). Write \(f_0=c\,x^a(1-x)^bP\) with integers \(a,b\ge0\) and \(P(0)P(1)\ne0\). By Theorem~\ref{thm:split}, \(\pi\) splits iff \(\deg P\le1\). If \(\deg P=0\), \(f_0=h_{a+b+1,a+1}\). If \(\deg P=1\), then \(P(0),P(1)>0\) because \(f_0\ge0\) and \(P\) has no zero at the endpoints, and \(f_0\) is the density of Proposition~\ref{prop:twopoint} with \(k'=a+b+2\), \(r=a+1\), \(wr=P(0)\), \(v(k'-r)=P(1)\). Theorem~\ref{thm:extension} identifies \(p_0\) with an extension of that prior. The converse is Proposition~\ref{prop:twopoint} together with Theorem~\ref{thm:extension}.
\end{proof}


\section{Choice operators}\label{sec:choice}

A (singleton) social choice function takes a structured finite set and outputs a winner. Under Plott's path-independence condition \cite{Plott1973,Aleskerov2007} this is equivalent to ranking, consistently as the set grows; this motivated our term `choice priors'. We call a choice prior {\it pure} if the winner in the initial group \(P_{-k},\dots,P_{-1}\) depends only on the ranking permutation, that is, if it is given by a mapping \(C:\mathfrak{S}_k\to[k]\) assigning to a permutation the benchmark rank.

\begin{proposition}\label{prop:pure}
Since the ranking in the initial group is uniformly random, the law of \(R_0\) for a pure choice operator is
\begin{equation}\label{eq:pure}
\mathbb{P}[R_0=r]=\frac{w_r}{k!},\qquad r\in[k],
\end{equation}
for nonnegative integers \(w_1,\dots,w_k\) summing to \(k!\). Conversely, every such signature arises in this way.
\end{proposition}

\begin{proof} Indeed,
\(w_r\) counts the permutations mapped to \(r\). Conversely, assign \(w_r\) permutations to each \(r\).
\end{proof}

There are \(\binom{k!+k-1}{k-1}\) pure choice priors of degree \(k\). For instance, the tournament of the Introduction has signature \((0,16,8,0)\) of \(4!=24\).
\begin{remark}
A permutation viewed as a mapping has the number of cycles as a natural statistic, and tilting it leads to the Ewens and (also tilting the cycle sizes) to the Pitman--Yor families \cite{Pitman2006}.
For the two-parameter extension of the case \(\alpha=-1\) with nonlinear cycle-count predictor
see \cite{Gnedin2010species}. A permutation viewed as an order in the projective limit carries a single exchangeable measure \cite{Rubin1966}, though steps in the direction of partial exchangeability lead to parametric families \cite{Gnedin2010coherent,Gnedin2006constrained}. Transformed through choice operators, the exchangeable order produces a wide variety of fully exchangeable phenomena.
\end{remark}
\section{The Topp--Leone family of priors}\label{sec:TL}

The Topp--Leone density \cite{ToppLeone1955,KotzVandorp2004} with shape parameter \(\nu>0\) is
\begin{equation}\label{eq:TL}
f(x)=2\nu\,x^{\nu-1}(1-x)(2-x)^{\nu-1},\qquad F(x)=\bigl(1-(1-x)^2\bigr)^\nu,\qquad0<x<1.
\end{equation}
The Beta-tilted version in the convention \eqref{eq:tilt} is
\begin{equation}\label{eq:TLtilt}
f(x\mid\alpha,\beta)\propto x^{\nu+\alpha-1}(1-x)^{\beta+1}(2-x)^{\nu-1},\qquad\alpha>-\nu,\ \beta>-2.
\end{equation}
For non-integer \(\nu\), \eqref{eq:TL} is not a polynomial, hence not a choice prior of finite degree, and its binomial moments are 
only expressible via the hypergeometric function.

For integer \(\nu\),   \(\mathrm{TL}(\nu)\) is induced by the two-stage tournament of Example~\ref{ex:twostage} with \(k_1=\nu\), \(k_2=2\): the players are organised in \(\nu\) pairs, the stronger player (smaller mark) of each pair is retained, and the benchmark is the weakest of these \(\nu\) pairwise winners. The binomial expansion of \((2-x)^{\nu-1}\) gives
\begin{equation}\label{eq:mTL}
m(A,B)=2\nu\sum_{i=0}^{\nu-1}\binom{\nu-1}{i}2^{\nu-1-i}(-1)^i\,\mathrm{B}(A+\nu+i,B+2),
\end{equation}
and for the tilted density \((A,B)\) is replaced by \((A+\alpha,B+\beta)\). By Proposition~\ref{prop:boundary} with \(a=\nu-1\), the boundary order at \(1\) is \(1\), and \(d=\nu-1\).

We have seen that splitting in linear factors occurs for
 \(\nu=1, 2\); but for larger $\nu$ it does not. In particular, for $\nu=3$
\begin{align}
\pi(A,B)&=\frac{(A+3)\bigl(A^2+4AB+4B^2+17A+36B+76\bigr)}{(A+B+7)\bigl(A^2+4AB+4B^2+15A+32B+60\bigr)}\,. ~~~~\label{eq:pi3}
\end{align}

\section{Limit theorems and diffusion approximations}\label{sec:diff}

We turn to the asymptotics. Both processes considered in this section live on the Pascal graph. The first is the walk \((A,B)\) itself, equivalently the running rank: its projective coordinate \(A/n\) converges to the benchmark location, and we describe the fluctuations about the limit, first for a fixed initial group as \(n\) grows, then for a large initial group. The second is the up-down chain on a fixed level $n$ 
of the graph, in which old outcomes are forgotten and new tests are made; under suitable scaling it converges to a Wright--Fisher diffusion whose drift is determined by the prior.

\subsection{The law of large numbers}\label{sec:LLN}

The almost sure convergence of frequencies in an exchangeable sequence is an instance of the ergodic theorem, since exchangeability implies stationarity.
In the present setting it follows directly from \eqref{eq:Rn} and the strong law of large numbers applied conditionally on \(X\):
\begin{equation}\label{eq:LLN}
\frac{R_n}{n}\to X\quad\text{a.s.\ as }n\to\infty.
\end{equation}
Let \(\mathcal{G}_n\) be generated by \(R_0\) and the first \(n\) outcomes, and set \(N_n:=k+n+1\) and \(Z_n:=R_n/N_n\).

\begin{proposition}\label{prop:mart}
The process \((Z_n)\) is a \((\mathcal{G}_n)\)-martingale with \(Z_n\to X\) a.s., and
\begin{align}
\mathbb{E}\bigl[(Z_{n+1}-Z_n)^2\mid\mathcal{G}_n\bigr]&=\frac{Z_n(1-Z_n)}{(N_n+1)^2},\label{eq:qv}\\
\mathbb{E}\bigl[Z_{n+1}(1-Z_{n+1})\mid\mathcal{G}_n\bigr]&=Z_n(1-Z_n)\Bigl(1-\frac{1}{(N_n+1)^2}\Bigr).\label{eq:Z1Z}
\end{align}
\end{proposition}

\begin{proof}
By Proposition~\ref{prop:urn}, \(Z_{n+1}=(R_n+1)/(N_n+1)\) with conditional probability \(Z_n\) and \(Z_{n+1}=R_n/(N_n+1)\) otherwise. This gives the martingale property and \eqref{eq:qv}, and \eqref{eq:Z1Z} follows from \(\mathbb{E}[Z_{n+1}^2\mid\mathcal{G}_n]=Z_n^2+Z_n(1-Z_n)(N_n+1)^{-2}\). Since \(Z_n-R_n/n\to0\), the limit is \(X\) by \eqref{eq:LLN}.
\end{proof}

Iterating \eqref{eq:Z1Z} and using \(\prod_{m\ge k+2}(1-m^{-2})=(k+1)/(k+2)\), we find \(\mathbb{E}[X(1-X)\mid R_0]=Z_0(1-Z_0)\,(k+1)/(k+2)\) and \(\mathrm{Var}[X\mid R_0]=Z_0(1-Z_0)/(k+2)\), in accord with \(X\mid R_0\stackrel{d}{=}\mathrm{Beta}(R_0,k+1-R_0)\). The identity \eqref{eq:Z1Z} is the discrete counterpart of the exponential decay of \(\mathbb{E}[Z(1-Z)]\) along the neutral Wright--Fisher diffusion with generator \(\tfrac12x(1-x)\varphi''\), and \eqref{eq:qv} says that \(Z\) has the infinitesimal variance of that diffusion, run in a time scale where level \(m\) of the Pascal graph is passed at rate \(m^{-2}\). This suggests to look at the levels projectively.

\subsection{Fluctuations of the running rank}

Fix \(n\) and parametrise the levels \(m\ge n\) of the Pascal graph by the ratio \(s=n/m\in(0,1]\), so that \(s\downarrow0\) corresponds to the infinite future. Let
\[
\xi_n(s):=\sqrt n\,\bigl(Z_{\lceil n/s\rceil}-X\bigr),\quad s\in(0,1],\qquad \xi_n(0):=0.
\]

\begin{proposition}\label{prop:FCLTn}
For fixed \(k\) and \(p_0\), as \(n\to\infty\), conditionally on \(X\) the process \(\xi_n\) converges in distribution in the Skorokhod space \(D[0,1]\) to \(\sqrt{X(1-X)}\,W\), where \(W\) is a standard Brownian motion independent of \(X\).
\end{proposition}

\begin{proof}
Write \(m=\lceil n/s\rceil\). Since \(Z_m-A_m/m=(R_0-(k+1)A_m/m)/N_m\), we have \(\sqrt n\,|Z_m-A_m/m|\le(k+1)/\sqrt n\), and we may replace \(Z_m\) by \(A_m/m\). Given \(X\), the outcomes are i.i.d.\ Bernoulli\((X)\). By Donsker's invariance principle  \(V_n(t):=n^{-1/2}(A_{\lfloor nt\rfloor}-\lfloor nt\rfloor X)\) converges to \(\sqrt{X(1-X)}\,W'(t)\) on every interval \([1,T]\), and \(\sqrt n(A_m/m-X)=(n/m)V_n(m/n)\) converges on \([1/T,1]\) to \(s\sqrt{X(1-X)}\,W'(1/s)\), which by time inversion is \(\sqrt{X(1-X)}\,W(s)\). Near \(s=0\), the sequence \((A_m/m-X)_{m\ge M}\) is a reversed martingale given \(X\), and by Doob's inequality
\[
\mathbb{E}\Bigl[\sup_{m\ge M}\Bigl(\frac{A_m}{m}-X\Bigr)^2\,\Big|\,X\Bigr]\le\frac{4X(1-X)}{M}.
\]
With \(M=\lceil n/h\rceil\) this gives \(\mathbb{E}[\sup_{s\le h}\xi_n(s)^2\mid X]\le h+o(1)\), and the approximation theorem of \cite{Billingsley1999} (Theorem 3.2) yields convergence in \(D[0,1]\).
\end{proof}

Thus, in the projective coordinate \(s\), the running rank approaches its limit like a Brownian motion run backwards from the infinite future, with the Wright--Fisher coefficient \(X(1-X)\) frozen at the limit; the result is a functional form of the mixed normal approximation $$\sqrt n(A_n/n-X)\Rightarrow\sqrt{X(1-X)}Z\,,~~~~Z\stackrel{d}{=}\mathcal N(0,1).$$

The same projective time scale serves the case of a large initial group, now with \(k+1\) in place of \(n\). Let
\[
u_n:=1-\frac{k+1}{N_n},\qquad n\ge0,
\]
which increases from \(u_0=0\) to \(1\) with steps \(\Delta u_n=(k+1)/\bigl(N_n(N_n+1)\bigr)\le(k+1)^{-1}\). For \(u\in[0,1)\) let \(n_k(u):=\max\{n:u_n\le u\}\), and set
\[
\zeta_k(u):=\sqrt{k+1}\,\bigl(Z_{n_k(u)}-Z_0\bigr),\quad u\in[0,1),\qquad \zeta_k(1):=\sqrt{k+1}\,(X-Z_0),
\]
a martingale on the closed interval \([0,1]\) with right-continuous paths having left limits.

\begin{proposition}\label{prop:FCLT}
Let \(k\to\infty\) with \(R_0/(k+1)\to z\in(0,1)\). Then \(\zeta_k\Rightarrow\sqrt{z(1-z)}\,W\) in \(D[0,1]\). In particular, \(\sqrt{k+1}\,(X-Z_0)\Rightarrow\mathcal{N}\bigl(0,z(1-z)\bigr)\).
\end{proposition}

\begin{proof}
Conditioning on \(R_0\), it suffices to consider deterministic \(R_0\) with \(Z_0\to z\).

\emph{Interior.} Fix \(u_1<1\). The pair \((u_n,\zeta_k(u_n))\) is a Markov chain with time steps \(\Delta u_n\le(k+1)^{-1}\). By Proposition~\ref{prop:urn}, \(\mathbb{E}[\Delta\zeta\mid\mathcal{G}_n]=0\), \(|\Delta\zeta|\le(k+1)^{-1/2}\), and by \eqref{eq:qv}, \(\mathbb{E}[(\Delta\zeta)^2\mid\mathcal{G}_n]=\Delta u_n\,Z_n(1-Z_n)\,N_n/(N_n+1)\). For \(\varphi\in C^3_c(\mathbb{R})\), Taylor's formula gives
\[
\frac{\mathbb{E}[\varphi(\zeta+\Delta\zeta)-\varphi(\zeta)\mid\mathcal{G}_n]}{\Delta u_n}
=\tfrac12\,Z_n(1-Z_n)\,\varphi''(\zeta)+O\bigl((k+1)^{-1/2}\|\varphi'''\|+N_n^{-1}\|\varphi''\|\bigr),
\]
and since \(Z_n=Z_0+\zeta/\sqrt{k+1}\), the coefficient \(Z_n(1-Z_n)\) converges to \(z(1-z)\) uniformly for \(\zeta\) in the support of \(\varphi\). The generators of the time-scaled chains thus converge, uniformly on \([0,u_1]\), to \(\tfrac12z(1-z)\,\partial_\zeta^2\), whose martingale problem is well posed, and the convergence theorem for Markov chains \cite{EthierKurtz1986} (Chapter 4, Section 8) gives \(\zeta_k\Rightarrow\sqrt{z(1-z)}\,W\) in \(D[0,u_1]\).

\emph{Endpoint.} Let \(h\in(0,1)\) and \(n=n_k(1-h)\), so that \(h\le(k+1)/N_n<h+(k+1)^{-1}\). By Doob's inequality, \eqref{eq:qv} and \(Z(1-Z)\le\frac14\),
\[
\mathbb{E}\sup_{v\in[1-h,1]}|\zeta_k(v)-\zeta_k(1-h)|^2\le(k+1)\sum_{m>N_n}m^{-2}\le\frac{k+1}{N_n}<h+\frac1{k+1}.
\]
The Skorokhod distance between \(\zeta_k\) and the stopped process \(\zeta_k(\cdot\wedge(1-h))\) is at most the supremum on the left, the stopped processes converge by the first step, and \(W(\cdot\wedge(1-h))\to W\) as \(h\to0\); Theorem 3.2 of \cite{Billingsley1999} yields convergence in \(D[0,1]\).
\end{proof}

The limit variance at \(u=1\) agrees with \(\mathrm{Var}[X\mid R_0]\) computed above, and the proposition is the functional form of the normal approximation of \(\mathrm{Beta}(R_0,k+1-R_0)\). In contrast to \cite{GnedinSeksenbayev2021,FarberGnedinMannan2023}, where the limit is a bridge and the drift is singular at the terminal time, here the coefficients stay bounded, and the endpoint is controlled by a single maximal inequality.

\subsection{The up-down chain}\label{sec:updown}

The forward walk on the Pascal graph leaves every level after one step. To keep the level fixed, we let the observer forget old outcomes. Let \(f\) be a probability density on \([0,1]\), with binomial moments \(m(A,B)\) and predictive rule \(\pi(A,B)=m(A+1,B)/m(A,B)\); by \eqref{eq:tilt}, a historic composition \((\alpha,\beta)\) may be absorbed into \(f\). The state of the chain on level \(n\) is the composition \((A,B)\), \(A+B=n\), of the outcomes of \(n\) players currently on record, and a step consists of
\begin{enumerate}[label=(\roman*)]
\item a \textbf{down move}: one of the \(n\) players on record, chosen uniformly at random, is removed together with his outcome;
\item an \textbf{up move}: a new player is tested, and wins with probability \(\pi\) evaluated at the reduced composition.
\end{enumerate}
By exchangeability, after the down move the remaining \(n-1\) outcomes are distributed as those of \(n-1\) tests, so the up move is the usual updating. The resulting birth--death chain has
\begin{align}
p_+(A,B)&:=\mathbb{P}[(A,B)\to(A+1,B-1)]=\frac{B}{n}\,\pi(A,B-1),\label{eq:p+}\\
p_-(A,B)&:=\mathbb{P}[(A,B)\to(A-1,B+1)]=\frac{A}{n}\bigl(1-\pi(A-1,B)\bigr),\label{eq:p-}
\end{align}
and holds with probability \(1-p_+-p_-\).

\begin{lemma}[Reversibility]\label{lem:rev}
The up-down chain on level \(n\) is reversible with respect to the law \(\omega^{(n)}_A=\binom nA m(A,B)\) of the number of wins in \(n\) tests.
\end{lemma}

\begin{proof}
The consistency of path probabilities, \(m(A,B-1)=m(A+1,B-1)+m(A,B)\) (Remark~\ref{rem:deFinetti}; in analytic terms the difference recursion of Hausdorff's moment problem \cite{Hausdorff1923}), gives \(1-\pi(A,B-1)=m(A,B)/m(A,B-1)\). With \(\binom nA\frac Bn=\binom n{A+1}\frac{A+1}n=\binom{n-1}A\),
\[
\omega^{(n)}_A\,p_+(A,B)=\binom{n-1}{A}\frac{m(A,B)\,m(A+1,B-1)}{m(A,B-1)}=\omega^{(n)}_{A+1}\,p_-(A+1,B-1).
\]
\end{proof}

Reversibility is common for up-down chains of Diaconis and Fill \cite{DiaconisFill1990} and their relatives on branching graphs \cite{BorodinOlshanski2009,DongGnedinPitman2007,GnedinPitman2006,Petrov2013}; the Pascal graph with Beta-type measures is the running example of \cite{Petrov2013}. If \(f\) is continuous at \(x\in(0,1)\), then \((n+1)\,\omega^{(n)}_{\lfloor xn\rfloor}\to f(x)\): indeed \((n+1)\omega^{(n)}_A\) is the integral of \(f\) against the \(\mathrm{Beta}(A+1,B+1)\) density, which concentrates at \(x\) for \(A=\lfloor xn\rfloor\). So the stationary law on level \(n\), in the projective coordinate \(A/n\), approximates the prior density.

\subsection{The diffusion limit}

The drift of the limit diffusion is read off a perturbation formula for the predictive rule. For the rest of the section let
\[
f(x)=c\,x^a(1-x)^bP(x),\qquad a,b>-1,
\]
with \(P\) positive on \([0,1]\), as in Proposition~\ref{prop:boundary} but with \(P\) only required to be smooth.

\begin{lemma}\label{lem:Stein}
Let \(P\in C^1[0,1]\). For a composition \((A,B)\) of \(n\) let \(T\stackrel{d}{=}\mathrm{Beta}(A+a+1,\,B+b+1)\) and \(\lambda:=n+a+b+2\). Then
\begin{equation}\label{eq:Stein}
\pi(A,B)=\frac{A+a+1}{\lambda}+\frac{1}{\lambda}\,\frac{\mathbb{E}\bigl[T(1-T)P'(T)\bigr]}{\mathbb{E}\,P(T)}.
\end{equation}
\end{lemma}

\begin{proof}
The posterior density is proportional to \(t^{A+a}(1-t)^{B+b}P(t)\), so \(\pi=\mathbb{E}[TP(T)]/\mathbb{E}P(T)=\mathbb{E}T+\mathrm{Cov}(T,P(T))/\mathbb{E}P(T)\). For \(T\stackrel{d}{=}\mathrm{Beta}(\alpha',\beta')\), integration by parts gives the Stein identity \(\mathbb{E}[(T-\mathbb{E}T)\,h(T)]=\mathbb{E}[T(1-T)h'(T)]/(\alpha'+\beta')\), which we apply with \(h=P\).
\end{proof}

The first term in \eqref{eq:Stein} is the rule of the Beta prior with the boundary exponents of \(f\); the cofactor contributes a correction of order \(1/n\), which survives in the drift. Let
\begin{equation}\label{eq:drift}
\mu_f(x):=(1+a)-(2+a+b)x+x(1-x)\frac{P'(x)}{P(x)}=(1-2x)+x(1-x)\bigl(\log f(x)\bigr)',
\end{equation}
and
\begin{equation}\label{eq:gen}
\mathcal{L}\varphi(x)=x(1-x)\varphi''(x)+\mu_f(x)\varphi'(x)=\frac1{f(x)}\bigl(x(1-x)f(x)\varphi'(x)\bigr)'.
\end{equation}

\begin{proposition}[Generator convergence]\label{prop:gen}
Let \(P\in C^2[0,1]\). Uniformly over the compositions \((A,B)\) of \(n\), with \(x=A/n\),
\[
n\bigl(p_+-p_-\bigr)=\mu_f(x)+O(n^{-1/2}),\qquad p_++p_-=2x(1-x)+O(n^{-1/2}),
\]
and for \(\varphi\in C^3[0,1]\) the rescaled generator \(\mathcal{L}_n\varphi(x):=n^2\,\mathbb{E}[\varphi(A'/n)-\varphi(A/n)]\), \(A'\) being the next state, satisfies \(\mathcal{L}_n\varphi=\mathcal{L}\varphi+O(n^{-1/2})\) uniformly on the grid.
\end{proposition}

\begin{proof}
The function \(t\mapsto t(1-t)P'(t)\) is Lipschitz and \(P\) is bounded away from \(0\). For \(T\) as in Lemma~\ref{lem:Stein} at the compositions \((A,B-1)\) and \((A-1,B)\) of \(n-1\), \(\mathbb{E}|T-x|=O(n^{-1/2})\) uniformly, so the correction term in \eqref{eq:Stein} equals \(\lambda^{-1}\bigl(x(1-x)P'(x)/P(x)+O(n^{-1/2})\bigr)\). Inserting \eqref{eq:Stein} into \eqref{eq:p+}--\eqref{eq:p-}, with \(\lambda=n+a+b+1\),
\[
n(p_+-p_-)=\frac{B(A+a+1)-A(B+b+1)}{\lambda}+\frac{n}{\lambda}\,x(1-x)\frac{P'(x)}{P(x)}+O(n^{-1/2}),
\]
and \(\bigl(B(a+1)-A(b+1)\bigr)/\lambda=(1-x)(1+a)-x(1+b)+O(n^{-1})\). The second assertion follows from \(p_\pm=x(1-x)+O(n^{-1/2})\). The state changes by at most one, so Taylor's formula gives \(\mathcal{L}_n\varphi=n(p_+-p_-)\varphi'+\tfrac12(p_++p_-)\varphi''+O(n^{-1}\|\varphi'''\|)\).
\end{proof}

\begin{theorem}[Diffusion limit]\label{thm:diff}
Let \(P\in C^2[0,1]\) be positive, and let \(A^{(n)}_t\) be the number of wins on record after \(t\) steps of the up-down chain on level \(n\), with \(A^{(n)}_0/n\to y_0\). Then \(Y_n(t):=n^{-1}A^{(n)}_{\lfloor n^2t\rfloor}\) converges in distribution in the Skorokhod space to the diffusion on \([0,1]\)
\begin{equation}\label{eq:SDE}
{\rm d}Y_t=\mu_f(Y_t)\,{\rm d}t+\sqrt{2Y_t(1-Y_t)}\,{\rm d}W_t,\qquad Y_0=y_0,
\end{equation}
which is reversible with respect to \(f\).
\end{theorem}

\begin{proof}
The drift \(\mu_f\) is Lipschitz, with \(\mu_f(0)=1+a>0\) and \(\mu_f(1)=-(1+b)<0\), and the diffusion coefficient is H\"older of order \(\tfrac12\). By the Yamada--Watanabe theorem \cite{YamadaWatanabe1971}, \eqref{eq:SDE} has pathwise unique solutions, which remain in \([0,1]\); weak solutions exist since the coefficients are continuous and bounded, pathwise uniqueness implies uniqueness in law, and solutions of the martingale problem for \(\mathcal{L}\) coincide with the laws of weak solutions \cite{EthierKurtz1986} (Chapter 5, Section 3). Hence the martingale problem for \(\mathcal{L}\) on \(C^3[0,1]\) is well posed. The state space is compact, and by Proposition~\ref{prop:gen} the generators converge uniformly, so the convergence theorem of \cite{EthierKurtz1986} (Chapter 4, Section 8) applies. Reversibility follows from the divergence form \eqref{eq:gen}.
\end{proof}

The diffusion coefficient \(2x(1-x)\) is universal, and the prior enters only through \((\log f)'\). Compared with the fluctuations of the running rank, the coefficient is doubled, because both the down move and the up move contribute to the variance, while the drift is produced by the up move alone. For the uniform prior \(\mathcal L\varphi=(x(1-x)\varphi')'\) is the Legendre operator on \([0,1]\): the shifted Legendre polynomials \(P_j(2x-1)\) are eigenfunctions with eigenvalues \(-j(j+1)\), and for Beta priors their role is taken by the Jacobi polynomials. Hausdorff \cite{Hausdorff1923} (p.~227) expresses his polynomial solution of the moment problem, discussed in Remark~\ref{rem:Hausdorff}, by means of the Legendre polynomials \(L_p(x)=P_p(2x-1)\); for a general choice prior the spectral decomposition of \(\mathcal L\) in \(L^2(f)\) is not explicit.

\subsection{The Topp--Leone family}\label{sec:TLdiff}

The tilted Topp--Leone density \eqref{eq:TLtilt} has the form of Theorem~\ref{thm:diff} with \(a=\nu+\alpha-1\), \(b=\beta+1\) and \(P=(2-x)^{\nu-1}\), for all real \(\nu>0\), \(\alpha>-\nu\), \(\beta>-2\). By \eqref{eq:drift},
\begin{equation}\label{eq:munu}
\mu_\nu(x;\alpha,\beta)=(\alpha+\nu)-(\alpha+\beta+\nu+2)x-(\nu-1)\frac{x(1-x)}{2-x}.
\end{equation}
For \(\nu=2,3\) the pre-limit chains are driven by the rules \eqref{TL2} and \eqref{eq:pi3}, for non-integer \(\nu\) by hypergeometric ones, and \eqref{eq:munu} holds in all cases. For \(\nu=1\), \(\mu_1(x;\alpha,\beta)=(\alpha+1)-(\alpha+\beta+3)x\) is the drift of the Wright--Fisher diffusion with immigration rates \(\alpha+1\) and \(\beta+2\); for \(\nu\ge2\) the additional term is a bounded perturbation vanishing at both endpoints.

Since the diffusion coefficient vanishes at the endpoints, the sign of the drift alone does not show where the stationary mass accumulates. By \eqref{eq:drift}, \((\log f_\nu)'=\bigl(\mu_\nu-(1-2x)\bigr)/\bigl(x(1-x)\bigr)\), so the stationary density increases where \(\mu_\nu>1-2x\). For instance, for \(\nu=1\), \(\alpha=\beta=0\), the drift \(1-3x\) is positive on \((0,\tfrac13)\), yet \(\mu_1-(1-2x)=-x<0\), and the stationary density \(2(1-x)\) is decreasing.

\begin{lemma}\label{lem:mode}
Let \(\nu>0\), \(\alpha>-\nu\), \(\beta>-1\).
\begin{enumerate}[label=(\alph*)]
\item If \(\nu+\alpha>1\), then \(f_\nu(\,\cdot\mid\alpha,\beta)\) has a unique mode in \((0,1)\).
\item If \(\nu\ge1\) and \(\nu+\alpha\le1\), then \(f_\nu(\,\cdot\mid\alpha,\beta)\) is strictly decreasing on \((0,1)\).
\end{enumerate}
\end{lemma}

\begin{proof}
On \((0,1)\), \((\log f_\nu)'=(\nu+\alpha-1)/x-(\beta+1)/(1-x)-(\nu-1)/(2-x)\). Multiplied by \(x(1-x)(2-x)\), this becomes the quadratic \(\psi(x)=e_2x^2+e_1x+e_0\) with \(e_2=2\nu+\alpha+\beta-1\), \(e_1=-4\nu-3\alpha-2\beta+2\), \(e_0=2\nu+2\alpha-2\). In (a), \(\psi(0)>0>\psi(1)=-(\beta+1)\), so \(\psi\) has a single root in \((0,1)\), where it changes sign from \(+\) to \(-\). In (b), all three terms of \((\log f_\nu)'\) are nonpositive, and the middle one is negative.
\end{proof}

In case (b) the density is unbounded at \(0\) when \(\nu+\alpha<1\). The condition \(\nu\ge1\) cannot be dropped: for \(\nu=\alpha=\tfrac12\), \(\beta=-0.9\), one has \(\nu+\alpha=1\), but \((\log f_\nu)'=-0.1/(1-x)+0.5/(2-x)\) changes sign at \(x=\tfrac34\), which is the mode.

\section{Ex-post choice and finite exchangeability}\label{sec:finite}

In the model of Section~\ref{sec:model} the benchmark is chosen among the initial group of players only. Adding a player then extends the uniform ranking, so the sequence of outcomes is always the restriction of an infinite exchangeable sequence, whose mixing density is a choice prior. A finite exchangeable sequence need not be of this kind. Its law is determined by the law \(q\) of the number of wins \(S_n\), and given \(S_n=s\) all arrangements of the outcomes are equally likely; so the sequence is a mixture over \(q\) of hypergeometric processes, that is of drawings without replacement from an urn with \(s\) winning and \(n-s\) losing tickets. Splitting by a benchmark simulates these drawings, if the choice operator is allowed to see the whole group of players, including those still to be tested.

\begin{proposition}\label{prop:expost}
Let \(q\) be a law on \(\{0,\dots,n\}\). Rank \(n+1\) players uniformly at random, declare the player of rank \(\rho\) to be the benchmark, where \(\rho-1\) has law \(q\), and test the other \(n\) players against the benchmark in uniformly random order. The outcomes are exchangeable with \(S_n\stackrel{d}{=}\rho-1\), and every exchangeable law on \(\{0,1\}^n\) arises in this way.
\end{proposition}

\begin{proof}
Given \(\rho-1=s\), exactly \(s\) of the other players are stronger than the benchmark, and their positions in the testing order form a uniformly random \(s\)-subset of \([n]\).
\end{proof}

The updating is again combinatorial. Given \(\rho-1=s\), a particular sequence with \(A\) wins and \(B\) losses in the first \(j=A+B\) tests has probability \([s]^\downarrow_A[n-s]^\downarrow_B/[n]^\downarrow_j\), whence
\[
\mathbb{P}[\rho-1=s\mid A,B]\propto q(s)\binom{s}{A}\binom{n-s}{B},\qquad \pi(A,B)=\frac{\mathbb{E}[\rho-1-A\mid A,B]}{n-j},
\]
since \(s-A\) of the remaining \(n-j\) players are stronger than the benchmark.

\begin{example}\label{ex:median}
For \(n=2\), let the benchmark be the median of the three players. Exactly one of the two newcomers wins: \(\pi(0,0)=\tfrac12\), \(\pi(1,0)=0\), \(\pi(0,1)=1\). This law is not the restriction of an exchangeable law of three trials. In contrast, the median of a past group of three gives the \(\mathrm{Beta}(2,2)\) urn (Example~\ref{ex:three}).
\end{example}

Extendibility is governed by the hypergeometric projection: a law \(q_N\) on \(\{0,\dots,N\}\), \(N>n\), projects to the law
\[
q(s)=\sum_{t}q_N(t)\,\frac{\binom ts\binom{N-t}{n-s}}{\binom Nn},\qquad s=0,\dots,n,
\]
of the number of wins among \(n\) of \(N\) exchangeable trials. The law of Proposition~\ref{prop:expost} extends to \(N\) trials if and only if \(q\) is such a projection; in benchmark terms, the group of \(n+1\) players is then a uniformly chosen subgroup, containing the benchmark, of a group of \(N+1\) players in which the benchmark rank minus one has law \(q_N\). Extendibility to all \(N\) characterises the mixtures of binomial laws \cite{DiaconisFreedman1980}, and for the median of three already \(N=3\) is impossible: \(q=(0,1,0)\) would require \(q_3(0)=q_3(1)=0\) and \(q_3(2)=q_3(3)=0\).

In the last example the predictive probability equals \(0\) or \(1\) at some vertices of the Pascal graph. This is the degenerate case of de Finetti \cite{deFinetti1952}, who showed that degeneracy at a vertex propagates along half-lines of the Pascal graph, and determined which configurations of degenerate vertices are compatible with indefinite extension. The harmonic law of Section~\ref{sec:improper} gives another instance.

Two operations on the group of players have thus appeared: growing it by unobserved past players, which is degree elevation (Theorem~\ref{thm:extension}), and embedding it in a larger group with unobserved future players, which is finite extendibility. The forward model, where the choice depends on the past only, is extendible to every horizon.

\section{Generalisations}\label{sec:general}

We close with two departures from the main setting: improper priors, which still produce explicit finite exchangeable laws, and statistics of the testing process other than the outcomes, which are only partially exchangeable but equally explicit.

\subsection{Improper priors}\label{sec:improper}

A pole of the mixing measure at \(0\) is familiar from Lévy measures, and measures \(\nu\) on \((0,1]\) with \(\int x\,\nu({\rm d}x)<\infty\) play the role of priors in the theory of regenerative composition structures \cite{GnedinPitman2005,Gnedin2004three}. On the Pascal graph, such a measure defines the function \(m_\nu(A,B)=\int x^A(1-x)^B\nu({\rm d}x)\), finite for \(A\ge1\), equal to \(+\infty\) on the line \(A=0\), and harmonic in the sense of the consistency relation of Remark~\ref{rem:deFinetti} wherever finite. Harmonic functions of this kind are called semifinite. For the Pascal graph, Vershik and Kerov describe them in the language of traces on the dimension group, which is the ring of polynomials with the basis \(p^k(1-p)^l\) on level \(k+l\) (Section~6 of \cite{VershikKerov1987}, p.~1728). Besides the finite indecomposable traces \(Q\mapsto Q(p)\), \(0\le p\le1\), which correspond to the Bernoulli laws, the indecomposable semifinite traces are the derivatives \(Q\mapsto Q^{(m)}(p)\) at \(p\in\{0,1\}\), \(m\ge1\), finite on the polynomials vanishing to order \(m\) at \(p\). For \(p=0\), on the vertices this is \(\psi(A,B)=m!\,1\{A=m\}\) for \(A\ge m\), and \(+\infty\) for \(A<m\): the limit of \(\varepsilon^{-m}\delta_\varepsilon\) as \(\varepsilon\downarrow0\), a pole of order \(m\) concentrated at the boundary. The functions \(m_\nu\) considered below are not indecomposable: they are mixtures of the finite traces \(Q\mapsto Q(p)\) with an infinite mixing measure \(\nu\). For a general theory of semifinite harmonic functions on branching graphs see \cite{Safonkin2021}.

In the benchmark model the pole arises by negative tilting (Proposition~\ref{prop:negtilt}).

\begin{proposition}\label{prop:semifinite}
If \(f_0\in\mathcal{H}_k\) has \(p_0(1)>0\), then \(\nu({\rm d}x)=x^{-1}f_0(x)\,{\rm d}x\) is infinite, \(m_\nu(A,B)=m(A-1,B)\) for \(A\ge1\), and \(m_\nu(A+1,B)/m_\nu(A,B)=\pi(A-1,B)\).
\end{proposition}

\begin{proof}
By Proposition~\ref{prop:negtilt}, \(\nu\) has the non-integrable term \(p_0(1)\,k\,x^{-1}(1-x)^{k-1}\); the rest is immediate.
\end{proof}

So \(x^{-1}f_0\) is a proper prior formally deprived of one win, and \(x^{-j}f_0\) is deprived of \(j\) wins; poles at \(1\) correspond to deprived losses. For \(k=1\), \(\nu({\rm d}x)=x^{-1}{\rm d}x\) is the uniform prior of Bayes deprived of one win, and \(x^{-1}(1-x)^{-1}{\rm d}x\), commonly called Haldane's prior, is the uniform prior deprived of one win and one loss.

\begin{remark}[A historic precedent]
In the appendix to Bayes's essay, Price argues that the first throw of a die of unknown constitution only shows that it has the side then thrown: ``after the first throw and not before, we should be in the circumstances required by the conditions of the present problem \dots\ With the second trial then our calculations must begin'' \cite{Bayes1763}. In the present terms, Price's prior before the first trial is the semifinite measure \(x^{-1}{\rm d}x\), whose posterior after one win is the uniform prior of Bayes.
\end{remark}

A semifinite measure does not define a law on infinite sequences, but it defines finite exchangeable laws by conditioning on at least one win.

\begin{proposition}\label{prop:harmonic}
Let \(\nu({\rm d}x)=x^{-1}{\rm d}x\). Conditioning \(\int\mathrm{Bernoulli}(x)^{\otimes n}\,\nu({\rm d}x)\) on \(\{S_n\ge1\}\) gives the exchangeable law with
\[
\mathbb{P}[S_n=s]=\frac{1}{s\,H_n},\qquad s=1,\dots,n,\qquad H_n=\sum_{i=1}^n\frac1i.
\]
For \(n\ge3\) this law is not the restriction of an exchangeable law on \(\{0,1\}^{n+1}\).
\end{proposition}

\begin{proof}
\(\binom{n}{s}\int_0^1x^{s-1}(1-x)^{n-s}{\rm d}x=1/s\) for \(s\ge1\). If \(q'\) on \(\{0,\dots,n+1\}\) projects to \(q(s)=1/(sH_n)\), then \(q(0)=0\) forces \(q'(0)=q'(1)=0\), and the projection equations
\(q(s)=q'(s)\frac{n+1-s}{n+1}+q'(s+1)\frac{s+1}{n+1}\) give \(q'(2)=\frac{n+1}{2}q(1)\) and \(q'(3)=\frac{(n+1)(2-n)}{6H_n}<0\) for \(n\ge3\).
\end{proof}

For \(n=2\), \(q=(0,\tfrac23,\tfrac13)\) is the projection of the law of three trials with exactly two wins. The law of Proposition~\ref{prop:harmonic} is familiar from the theory of regenerative compositions: it is the law of the size of the first block in a sample of size \(n\) for the measure \(z^{-1}(1-z)^{\theta-1}{\rm d}z\) with \(\theta=1\), the normalising constant being in general the generalised harmonic number \(\sum_{j=1}^n1/(\theta+j-1)\) \cite{Gnedin2004three}. For integer \(\theta=k\), this measure is \(k^{-1}x^{-1}h_{k,1}(x)\,{\rm d}x\): the benchmark is the strongest of \(k\) past players, deprived of one win. In both theories the denominators are the binomial moments of the tail, as in the following statement.

\begin{proposition}\label{prop:semipred}
Let \(f_0\) be a probability density, \(\nu({\rm d}x)=x^{-1}f_0(x)\,{\rm d}x\), and condition the exchangeable measure with mixing measure \(\nu\) on \(\{S_n\ge1\}\). Then for \(A+B<n\)
\[
\pi^{(n)}(A,B)=\frac{m(A,B)}{m(A-1,B)}\quad(A\ge1),\qquad
\pi^{(n)}(0,B)=\frac{m(0,B)}{\sum_{i=B}^{n-1}m(0,i)}.
\]
\end{proposition}

\begin{proof}
The mass of a prefix with \(A\ge1\) wins and \(B\) losses is \(m(A-1,B)\). For \(A=0\) it is \(\int x^{-1}(1-x)^B\bigl(1-(1-x)^{n-B}\bigr)f_0(x)\,{\rm d}x=\sum_{i=B}^{n-1}m(0,i)\), because \((1-(1-x)^j)/x=\sum_{i<j}(1-x)^i\).
\end{proof}

After the first win the rule is that of \(f_0\) shifted by one win, and does not depend on \(n\). Before the first win, since \(m(0,i)\) is the probability that the first win under \(f_0\) occurs after time \(i\), the rule is the hazard rate of the equilibrium (stationary excess) law of that waiting time, truncated at \(n\). For \(f_0=1\), \(\pi^{(n)}(A,B)=A/(A+B+1)\) for \(A\ge1\) and \(\pi^{(n)}(0,B)=(B+1)^{-1}/(H_n-H_B)\). The limit \(n\to\infty\) is proper if and only if \(\int x^{-1}f_0<\infty\), which for a choice prior means \(p_0(1)=0\), in accord with Proposition~\ref{prop:negtilt}.

\subsection{Partial exchangeability}\label{sec:nonexch}

The ranking rule produces, besides the outcomes of the tests, other statistics of the testing process. Their laws are not exchangeable, but they are partially exchangeable in the following sense: the counts remain a sufficient statistic, and the probability of a path on the lattice to a given vertex is a function of the vertex times a product of edge weights of a weighted Pascal graph. Thus the forward transition probabilities depend only on the current vertex, as for exchangeable sequences, while the backward (lookback) transitions are driven by weighted path counts instead of uniform ones. Such processes are called lookback similar in \cite{Gnedin2011boundaries}, where their boundaries are described for several number triangles. Two statistics of the testing procedure with \(k=1\) give the Stirling triangles.

\begin{example}[Neighbours of the benchmark]\label{ex:ewens}
Call trial \(j\) a contact if \(P_j\) settles in one of the two slots adjacent to the benchmark. When \(P_j\) arrives there are \(j\) players and \(j+1\) slots, so by the ranking rule the contacts occur independently with probability \(2/(j+1)\): these are Karamata--Stirling trials with parameter \(\theta=2\). A particular path with \(K\) contacts in \(n\) trials has probability
\[
\frac{2^K}{(n+1)!}\prod_{j\ \text{no contact}}(j-1).
\]
The first factor depends on the vertex \((n,K)\) only and is the Ewens\((2)\) probability of a single permutation of \([n]\) with \(K\) cycles; the edge weights \(j-1\) are those of the Stirling triangle. Summing over paths, \(\mathbb{P}[K_n=K]=2^K\left[{n\atop K}\right]/(n+1)!\) with unsigned Stirling numbers of the first kind, the law of the number of cycles of an Ewens\((2)\) permutation \cite{Pitman2006}, and a new contact occurs next with probability \(2/(n+2)\), whatever \(K\). Bayes's construction thus already contains the Ewens family.
\end{example}

\begin{example}[Records]
Let every winner replace the benchmark. Then \(P_j\) wins if and only if it is stronger than all of \(P_{-1},P_1,\dots,P_{j-1}\), that is, the benchmark updates are the records. By the ranking rule they occur independently with probability \(1/(j+1)\), which are the Karamata--Stirling trials with \(\theta=1\). A path with \(K\) records has probability \(\prod_{j\ \text{no record}}j\,/(n+1)!\), and the number of records has the law \(\left[{n+1\atop K+1}\right]/(n+1)!\) of the number of cycles of a uniform permutation of \([n+1]\) less one. Replacing the benchmark lowers the parameter from \(\theta=2\) to \(\theta=1\).
\end{example}

\end{document}